\documentclass[reqno]{amsart}

\title[Dual Transport Problem]{A Generalized Dual Maximizer for the Monge--Kantorovich Transport Problem}

\date{September 2010}

\author{Mathias Beiglb\"ock}
\address{University of Vienna. Faculty of Mathematics. Nordbergstrasse 15. 1090 Vienna, Austria}
\email{mathias.beiglboeck@univie.ac.at}
\author{Christian L\'eonard}
\address{Modal-X, Universit\'e Paris Ouest. B\^at. G, 200 av. de la R\'epublique. 92001 Nanterre, France}
\email{christian.leonard@u-paris10.fr}
\author{Walter Schachermayer}
\address{University of Vienna. Faculty of Mathematics. Nordbergstrasse 15. 1090 Vienna, Austria}
\email{walter.schachermayer@univie.ac.at}

\thanks{The first author acknowledges financial support from the Austrian Science
Fund (FWF) under grant P21209. The third author acknowledges
support from the Austrian Science Fund (FWF) under grant P19456,
from the Vienna Science and Technology Fund (WWTF) under grant
MA13 and by the Christian Doppler Research Association (CDG). All
authors thank A.~Pratelli for helpful discussions on the topic of
this paper. We also thank  M.~Goldstern and
G.~Maresch for their advice.}

 \subjclass[2000]{}

\usepackage{amsmath,amssymb,amsthm}
\usepackage{bbm}
\usepackage{enumerate}

\renewcommand{\phi}{\varphi}

\newcommand{\pphi}{\hat \phi}
\newcommand{\ppsi}{\hat \psi}
\newcommand{\fop}{\phi\oplus\psi}
\newcommand{\fopp}{\pphi\oplus\ppsi}
\newcommand{\fopn}{\phi_n\oplus\psi_n}
\newcommand{\ppi}{\hat \pi}
\newcommand{\hh}{\hat{h}}

\newcommand{\XY}{{X\times Y}}

\newcommand\R{\mathbb{R}}
\newcommand\N{\mathbb{N}}
\newcommand\Z{\mathbb{Z}}
\newcommand\I{\mathbf{1}}
\newcommand\1{\mathbbm{1}}

\newcommand{\Prel}{P^{\mathrm{{rel}}}}

\newcommand{\eps}{\varepsilon}

\newtheorem{theorem}{Theorem}[section]

\newtheorem{proposition}[theorem]{Proposition}

\theoremstyle{definition}
\newtheorem{example}[theorem]{Example}

\begin{document}

\begin{abstract}
The dual attainment of the Monge--Kantorovich transport problem is
analyzed in a general setting. The spaces $X, Y$ are assumed to be
polish and equipped with Borel probability measures $\mu$ and
$\nu$. The transport cost function $c:\XY \to [0,\infty]$ is
assumed to be Borel measurable. We show that a dual optimizer
always exists, provided we interpret it as a projective limit of
certain finitely additive measures. Our methods are functional
analytic and rely on Fenchel's perturbation technique.
\end{abstract}

\maketitle
\section{Introduction}

We consider the \emph{Monge-Kantorovich transport problem} for
Borel probability measures  $\mu,\nu$ on polish spaces $X,Y$. See
\cite{Vill03,Vill09} for  an excellent  account of the theory of
optimal transportation. The set $\Pi(\mu,\nu)$ consists of all
Monge-Kantorovich \emph{transport plans}, that is,  Borel
probability measures on $\XY$ which have $X$-marginal $\mu$ and
$Y$-marginal $\nu$. The \emph{transport costs} associated to a
transport plan $\pi$ are given by
\begin{equation}\label{CostFunctional}
    \langle c,\pi\rangle =\int_{\XY}c(x,y)\,d\pi(x,y).
\end{equation}
In most applications of the theory of optimal transport, the cost
function $c:\XY\to [0,\infty]$ is lower semicontinuous and only
takes values in $\R_+.$ But equation (\ref{CostFunctional}) makes
perfect sense if the $[0,\infty]$-valued cost function only is
Borel measurable. We therefore assume throughout this paper that
$c:\XY\to [0,\infty]$ is a Borel measurable function which may
very well assume the value $+ \infty$ for ``many'' $(x,y) \in
\XY$. The subset $\{c=\infty\}$ of $\XY$ is a set of forbidden
transitions.

Optimal transport on the Wiener space
\cite{FeUs02,FeUs04a,FeUs04b,FeUs06}) and on configuration spaces
\cite{Dec08, DJS08} provide natural infinite dimensional settings
where $c$ takes infinite values.

\medskip

The (primal) Monge-Kantorovich problem is to determine the primal
value
\begin{equation}
P:=\inf\{ \langle c, \pi\rangle:\pi\in \Pi(\mu,\nu)\} 
\label{G1}
\end{equation}
and to identify a primal optimizer $\hat{\pi} \in \Pi(\mu,\nu)$
which is also called an \emph{optimal transport plan}. Clearly,
without loss of generality this minimization can be performed
among the \emph{finite transport plans}, i.e.\ the infimum is
taken over the plans $\pi\in\Pi(\mu,\nu)$ verifying $\langle c,
\pi\rangle<\infty.$

The dual Monge-Kantorovich problem  consists in determining

\begin{equation}\label{SimpleJ}
D:=\sup \left\{ \int_{X}\phi\,d\mu+\int_{Y}\psi\,d\nu  \right\}
\end{equation}
for $(\phi,\psi)$ varying over the set of pairs of functions
$\phi:X\to[-\infty,\infty)$ and $\psi:Y\to[-\infty,\infty)$ which
are \emph{integrable}, i.e.\ $\phi\in L^1(\mu),$ $\psi\in
L^1(\nu)$, and satisfy $\fop\leq c.$ We have denoted
$\fop(x,y):=\phi(x)+\psi(y),$ $x\in X,$ $y\in Y.$

We say that there is \emph{no duality gap} if the primal value $P$
of the problem equals the dual value $D,$ there is primal
attainment if there exists some optimal plan $\ppi$ and there is
\emph{integrable dual attainment} if the above dual
Monge-Kantorovich problem is attained for some $(\pphi,\ppsi)$.
There is a long line of research on these questions, initiated
already by Kantorovich (\cite{Kant42}) himself and continued by
numerous others (we mention
\cite{KaRu58,Dudl76,Dudl02,deAc82,GaRu81,Fern81,Szul82,RaRu95,RaRu96,Mika06,MiTh06},
see also the bibliographical notes in \cite[p 86, 87]{Vill09}).
Important progresses were done by Kellerer \cite{Kell84}. We also
refer to the seminal paper \cite{GaMc96} by Gangbo and McCann.
Recently the authors of the present article have obtained in
\cite{BLS09a} a general duality result which is recalled below at
Theorem \ref{res-01}.

It is well-known that there is primal attainment under the
assumptions that $c$ is lower semicontinuous and the primal value
$P$ is finite. On the other hand, it is easy to build examples
where $c$ is not lower semicontinuous and no primal minimizer
exists.

In this article we focus onto the question of the dual attainment.

The dual optimizers $(\pphi,\ppsi)$ are sometimes called
Kantorovich potentials. In the Euclidean case with a quadratic
cost, it is well-known that these potentials are convex conjugate
to each other and that any optimal plan is supported by the
subdifferential of $\pphi.$ In the general case, these potentials
are $c$-conjugate to each other, a notion introduced by
R\"uschendorf \cite{Rus96}.

Kellerer \cite[Theorem 2.21]{Kell84} established that integrable
dual attainment holds true in the case of bounded $c$. This was extended
by Ambrosio and Pratelli \cite[Theorem 3.2]{AP02}, who gave
appropriate moment conditions on $\mu$ and $\nu$ which are
sufficient to guarantee the existence of integrable dual
optimizers. Easy examples show that one cannot expect that the
dual problem admits integrable maximizers unless the cost function
satisfies certain integrability conditions with respect to $\mu$
and $\nu$ \cite[Examples 4.4, 4.5]{BeSc08}. In fact \cite[Example
4.5]{BeSc08} takes place in a very ``regular'' setting, where $c$
is squared Euclidean distance on $\R$. In this case there exist
natural candidates $(\pphi, \ppsi)$ for the dual optimizer which,
however, fail to be dual maximizers in the usual sense as they are
not integrable.

The following solution was proposed in \cite[Section 1.1]{BeSc08}.
If $\phi$ and $\psi$ are integrable functions and
$\pi\in\Pi(\mu,\nu)$ then
\begin{align}\label{G29}
\int_X \phi \,d\mu+\int_Y \psi\, d\nu=\int_{\XY} \fop\, d\pi .
\end{align}
If we drop the integrability condition on $\phi$ and $\psi$, the
left hand side need not make sense. But if we require that
$\fop\le c$ and if $\pi$ is a finite cost transport plan, i.e.\
$\int_{\XY} c\,d\pi <\infty$, then the right hand side of
(\ref{G29}) still makes good sense, assuming possibly the value
$-\infty$,  and we set
\begin{align*}
J_c(\phi,\psi)=\int_{\XY} \fop\, d\pi.
\end{align*}
It is not difficult to show (see \cite[Lemma 1.1]{BeSc08}) that
this value does not depend on the choice of the finite cost
transport plan $\pi$ and satisfies $J_c(\phi,\psi)\leq D$. Under
the assumption that there exists some finite transport plan, we
then say that we have \emph{ measurable dual attainment} in the
optimization problem \eqref{SimpleJ} if there exist Borel
measurable functions $\hat \phi:X\to [-\infty, \infty)$ and $ \hat
\psi: Y\to [-\infty, \infty)$ verifying $\fopp\leq c$ such that
\begin{align}\label{BetterJ}
&D=J_c(\pphi,\ppsi).
\end{align}

In \cite[Theorem 2]{BeSc08} it was shown that, for Borel
measurable $c:\XY\to [0,\infty]$ such that $c<\infty, \mu \otimes
\nu$-almost surely, there is no duality gap and there is
measurable dual attainment in the sense of \eqref{BetterJ}.

A necessary and sufficient condition for the measurable dual
attainment was proved in \cite[Theorems 1.2 and 3.5]{BLS09a}. We
need some more notation to state this result below as Theorem
\ref{res-01}. Fix $0\le\varepsilon\le 1$ and define
    $\Pi^{\eps}(\mu,\nu)=\{\pi\in\mathcal{M}_\XY^+, \|\pi\|\geq 1-\varepsilon , p_X(\pi)\le \mu, p_Y(\pi)\le \nu\}$
where $\mathcal{M}_\XY^+$ denotes the non-negative Borel measures
$\pi$ on $\XY$ with norm $\|\pi\| =\pi(\XY).$ By $p_X (\pi)\le\mu$
(resp. $p_Y (\pi)\le\nu$) we mean that the projection of $\pi$
onto $X$ (resp. onto $Y$) is dominated by $\mu$ (resp. $\nu$). We
denote
    $ 
    P^\varepsilon :=\inf\left\{\langle
c,\pi\rangle : \pi\in\Pi^{\eps}(\mu,\nu)\right\}.
    $ 
This partial transport problem has recently been studied by
Caffarelli and McCann \cite{CaMc06}  as well as Figalli
\cite{Figa09}. In their work the emphasis is on a finer analysis
of the Monge problem for the squared Euclidean distance on
$\mathbb{R}^n$, and pertains to a fixed $\varepsilon>0$. In the
present paper, we do not deal with these more subtle issues of the
Monge problem and always remain in the realm of the Kantorovich
problem (\ref{G1}).  We call
\begin{align}
\label{G8} \Prel := \lim\limits_{\varepsilon\to 0} P^\varepsilon
\end{align}
the relaxed primal value of the transport plan. Obviously this
limit exists (assuming possibly the value $+ ~\infty$) and
$\Prel\le P$.

\begin{theorem}[Measurable dual attainment \cite{BLS09a}]\label{res-01}
Let $X, Y$ be polish spaces, equipped with Borel probability
measures $\mu,\nu$, and let $c:\XY\to [0 ,\infty]$ be Borel
measurable.
\begin{enumerate}[(a)]
    \item There is no duality gap if the primal problem is defined in
the relaxed form (\ref{G8}) while the dual problem is formulated
in its usual form \eqref{SimpleJ}. In other words, we have $\Prel
=D.$
    \item Assume that in addition there exists a finite
transport plan $\pi\in\Pi(\mu,\nu)$. The following statements are
equivalent.
\begin{enumerate}
\item[(i)] There is measurable dual attainment, i.e.\ there exist
measurable functions $\pphi, \ppsi$ such that $\fopp\leq c$ and
$\Prel=D=J_c(\pphi,\ppsi)$. \item[(ii)] There exists a $\mu\otimes
\nu$-a.s.\ finite function $h:\XY \to [0,\infty]$ such that
$\Prel= P_{c\wedge h}:=\inf\{ \langle c\wedge h, \pi\rangle:\pi\in
\Pi(\mu,\nu)\}$.
\end{enumerate}
\end{enumerate}
\end{theorem}

The aim of the present paper is to go beyond the setting of this
theorem where the measurable dual attainment is realized. We are
going to discuss the existence of an optimizer of an extension of
the dual problem \eqref{SimpleJ}, without imposing any further
conditions on the Borel measurable cost function $c:\XY \to
[0,\infty]$.

In Theorem \ref{LeonardDuality} we take a somewhat unothodox view at the general optimization problem. We start with a transport plan $\pi_0\in \Pi(\mu,\nu)$ with finite cost, but which is \emph{not} supposed to be optimal. We then optimize over all the transport plans $\pi\in\Pi(\mu,\nu)$ such that the Radon-Nikodym derivative $\frac{d\pi}{d\pi_0}$ is bounded. In this setting we show that there is no duality gap and that there is a dual optimizer. However, this dual optimizer is not given by a pair of functions $(\phi\oplus\psi)\in L^1(\pi_0)$, but rather as a weak star limit of a sequence $(\phi_n\oplus\psi_n)_{n=1}^\infty\in L^1(\pi_0)$ in the bidual $L^1(\pi_0)^{**}$. A rather elaborate example in the accompanying paper \cite{BLS09b} shows that this passage to the bidual is indeed necessary, in general.

While Theorem \ref{LeonardDuality} depends on the choice of the finite transport plan $\pi_0\in \Pi(\mu,\nu)$, we formulate in Theorem \ref{GeneralDualityLimit} a result which does not depend on this choice. There we pass to a projective limit along a net of finite transport plans. Again we can prove that there is no duality gap and can identify a dual optimizer.

\section{Two types of accident}

In this section, we point out some difficulties which arise when
going one step beyond the measurable dual attainment. We shall
face two types of troubles which might be called
\begin{itemize}
    \item measurability accident;
    \item singular concentration accident.
\end{itemize}
Before describing these phenomena, it is worth recalling some
results from \cite{BeSc08} and \cite{Leo07b} about optimal plans.
The proofs of the present paper and of Theorems \ref{res-02} and
\ref{res-03} below rely on three different types of techniques.

\subsection*{About the optimal plans}

The following characterization of the optimal plans was proved in
\cite{BeSc08}.

\begin{theorem}[{\cite[Theorem 2]{BeSc08}}]\label{res-02}
Assume that $X, Y$ are polish spaces equipped with Borel
probability measures $\mu, \nu,$ that $c : \XY \to [0,\infty]$ is
Borel measurable and $\mu\otimes\nu$-a.e.\ finite and that there
exists a finite transport plan.
\begin{enumerate}[(a)]
    \item Let $\pi$ be a finite transport plan and assume that there exist
measurable functions $\phi:X\to[-\infty,\infty)$ and
$\psi:Y\to[-\infty,\infty)$  which satisfy
\begin{equation}\label{eq-01}
    \left\{\begin{array}{l}
       \phi\oplus \psi\leq c\quad \textrm{everywhere} \\
        \phi\oplus \psi= c\quad \pi\textrm{-almost everywhere.} \\
    \end{array}\right.
    \end{equation}
Then $J_c(\phi, \psi)=\langle c,\pi\rangle,$ thus $\pi$ is an
optimal transport plan and $\phi,\psi$ are dual maximizers in the
sense of \eqref{BetterJ}.
    \item Assume that $\ppi$ is an optimal transport plan. Then $\ppi$ verifies \eqref{eq-01} for every pair $(\pphi,\ppsi)$
of dual maximizers in the sense of \eqref{BetterJ}.
\end{enumerate}
\end{theorem}
As a definition which was introduced in \cite{ScTe08}, a transport
plan $\pi$ is said to be \emph{strongly $c$-cyclically monotone}
if there exist measurable functions $\phi : X \to
[-\infty,\infty), \psi : Y \to [-\infty,\infty)$ which satisfy
\eqref{eq-01}.

Denote  by $\Pi(\mu,\nu,c)$ the set of finite cost transport plans
\[ \Pi(\mu,\nu,c) := \left\{ \pi\in\Pi(\mu,\nu)  : \int_{\XY} c\,d\pi <\infty \right\}, \]
and say that a property holds $\Pi(\mu,\nu,c)$-almost everywhere
if it holds true outside a measurable set $N$ such that
$\pi(N)=0,$ for all $\pi\in\Pi(\mu,\nu,c).$

In \cite{Leo07b}, the assumption that $c$ is $\mu\otimes\nu$-a.e.\
finite was removed under the extra requirement that $c$ is lower
semicontinuous and the following analogous results were obtained.

\begin{theorem}[\cite{Leo07b}]\label{res-03}
Assume that $X, Y$ are polish spaces equipped with Borel
probability measures $\mu, \nu,$ that $c : \XY \to [0,\infty]$ is
lower semicontinuous and that there exists a finite transport
plan.
\begin{enumerate}[(a)]
    \item Let $\pi$ be a finite plan and assume that there exist
measurable functions $\phi:X\to[-\infty,\infty)$ and
$\psi:Y\to[-\infty,\infty)$  which satisfy
\begin{equation}\label{eq-02}
    \left\{\begin{array}{l}
       \phi\oplus \psi\leq c\quad \Pi(\mu,\nu,c)\textrm{-almost everywhere} \\
        \phi\oplus \psi= c\quad \pi\textrm{-almost everywhere.} \\
    \end{array}\right.
    \end{equation}
Then $J_c(\phi, \psi)=\langle c,\pi\rangle,$ thus $\pi$ is an
optimal transport plan and $\phi,\psi$ are dual maximizers  in the
sense of \eqref{BetterJ}.

    \item Take any optimal plan $\ppi,$  $\epsilon>0$ and $\pi_o$ any
probability measure on $\XY$ such that $\int_{\XY}
c\,d\pi_o<\infty.$ Then, there exist functions $h\in
L^1(\ppi+\pi_o),$ $\phi$ and $\psi$ bounded continuous on $X$ and
$Y$ respectively and a measurable subset $Z_\epsilon\subset (\XY)$
such that
\begin{enumerate}[(i)]
    \item
    $h= c,\ \ppi$-almost
    everywhere on $(\XY)\setminus Z_\epsilon;$
    \item
    $\int_{Z_\epsilon} (1+c)\,d\ppi\le\epsilon;$
    \item
    $-c/\epsilon\le h\le c,\ (\ppi+\pi_o)$-almost everywhere;
    \item
    $-c/\epsilon\le \phi\oplus\psi\le c,$ everywhere;
    \item
    $\|h-\phi\oplus\psi\|_{L^1(\ppi+\pi_o)}\le\epsilon.$
\end{enumerate}
\end{enumerate}
\end{theorem}

As regards (\textit{a}), the examples \cite[Example 5.1]{BGMS08}
and \cite[Example 4.2]{BeSc08} exhibit optimal plans which are not
strongly $c$-cyclically monotone but which satisfy the weaker
property \eqref{eq-02}. As regards (\textit{b}), let us emphasize
the appearance of the probability measure $\pi_o$ in items
(\textit{iii}) and (\textit{v}). One can read (\textit{iii-v}) as
an approximation of $\phi\oplus\psi\le c,$ $(\ppi+ \pi_o)$-a.e.
Since it is required that $\int_{\XY} c\,d\pi_o<\infty,$ one can
choose $\pi_o$ in $\Pi(\mu,\nu,c),$ and the properties
(\textit{i-v}) are an approximation of \eqref{eq-02} where
$\Pi(\mu,\nu,c)$-a.e.\! is replaced by the weaker
$(\ppi+\pi_o)$-a.e.

Note also that for any $(\phi,\psi)$ verifying \eqref{eq-01} or
\eqref{eq-02} with  $\pi\in\Pi(\mu,\nu,c)$,  we have
\begin{equation}\label{eq-03}
    \mu(\phi=-\infty)=\nu(\psi=-\infty)=0.
\end{equation}
As a consequence of this remark and a result of Kellerer
\cite{Kell84}, see \cite[Lemma A.1]{BLS09a}, we can
replace``$\fop\le c$ everywhere" in \eqref{eq-01} by ``$\fop\le
c,$ $\Pi(\mu,\nu)$-almost everywhere." The comparison between
\eqref{eq-01} and \eqref{eq-02} becomes clearer.

\subsection*{Measurability accident}
To develop a feeling for what we are after, we consider a specific
example.

\begin{example}[Ambrosio-Pratelli, {\cite[Example 3.2]{AP02}}]\label{SecondAP}
Let $X=Y= [0,1)$, equipped with Lebesgue measure
$\lambda=\mu=\nu$. Pick $\alpha\in [0,1)$ irrational. Set
$$\Gamma_0=\{(x,x):x\in X\}\quad\Gamma_1=\{(x,x\oplus\alpha):x\in
X\},$$ where $\oplus$ is addition modulo $1$. Define $c: \XY \to
[0,\infty]$ by
\begin{align*}
c(x,y)=\left\{
\begin{array}{cl}
1&\mbox{ for }(x,y)\in\Gamma_0\\
2&\mbox{ for }(x,y)\in\Gamma_1, x\in [0,1/2)\\
0&\mbox{ for }(x,y)\in\Gamma_1, x\in [1/2,1)\\
\infty&\mbox{ else }
\end{array}\right..
\end{align*}
This cost function is a variation on \cite{AP02}'s original
example which has been proposed in \cite[Example 4.3]{BeSc08}. For
$i=0,1$, let $\pi_i$ be the obvious transport plan supported by
$\Gamma_i$. Following the arguments of \cite{AP02}, it is
easy to see 
that all finite transport plans are
given by convex combinations of the form $\rho\pi_0+(1-\rho)\pi_1,
\rho\in[0,1]$ and each of these transport plans leads to costs of
$1$.
\\
Note that since $c$ is lower semicontinuous, there is no duality gap.
This was proved in \cite{Kell84} and is an easy consequence of
Theorem \ref{res-01}-(a). Thus, for each $\eps>0$, there are
integrable functions $\phi, \psi: [0,1)\to[-\infty,\infty)$ such
that $\fop\leq c$ and $0\le\int (c-\fop)\, d\pi_i\leq \eps$ for
$i=0,1$.
\\
On the other hand, it is shown in \cite{BeSc08} that there do not
exist \emph{measurable} functions $\phi, \psi:
[0,1)\to[-\infty,\infty)$ satisfying $\fop\leq c$ such that
$\fop=c$ holds $\pi_0$- as well as $\pi_1$-almost surely.
\end{example}

Let us have a closer look at the previous example: while it is
{\it not possible} to find Borel measurable limits
$\hat{\phi},\hat{\psi}$ of an optimizing sequence
$(\phi_n,\psi_n)^\infty_{n=1}$, it {\it is possible} to find a
limiting Borel function $\hh(x,y)$ of the sequence of functions
$(\phi_n(x)+\psi_n(y))^\infty_{n=1}$ on the set $\{(x,y)\in \XY
:c(x,y) <\infty\}$. Indeed, on this set, which simply equals
$\Gamma_0 \cup\Gamma_1$, any optimizing sequence
$(\phi_n(x)+\psi_n(y))^\infty_{n=1}$ for \eqref{SimpleJ} has a
subsequence which converges $\pi$-a.s.\ to $\hh(x,y):=c(x,y)$, for
any finite cost transport plan $\pi$.

Summing up: in the context of the previous example, there is a
Borel function $\hh(x,y)$ on $\XY$, which equals $c(x,y)$ on
$\Gamma_0\cup\Gamma_1$; it may take any value on $(\XY)\setminus
(\Gamma_0\cup\Gamma_1)$, e.g.\ the value $+\infty$. This function
$\hh(x,y)$ may be considered as a kind of dual optimizer: it is,
for any finite cost transport plan $\pi$, the limit of an
optimizing sequence $(\phi_n(x)+\psi_n(y))^\infty_{n=1}$ with
respect to the norm $\|\cdot\|_{L^1(\pi)}$.

\subsection*{Singular concentration accident}
One can rewrite the sufficient conditions of Theorems
\ref{res-02}-(a) and \ref{res-03}-(a) as follows: $\ppi$ and
$(\pphi,\ppsi)$ solve the primal and dual problems if
$\ppi\in\Pi(\mu,\nu,c),$ $(\fopp) \ppi=c\ppi$ and $(\fopp)\pi\le
c\pi,$ $\forall\pi\in\Pi(\mu,\nu,c),$ in the space of bounded
measures. In view of Example \ref{SecondAP} and of part (b) of
Theorem \ref{res-03}, we are aware that $\fopp$ should be replaced
by a jointly measurable $\hh$ such that for each
$\pi\in\Pi(\mu,\nu,c),$ $\hh\pi$ can be approximated in variation
norm by a sequence $((\phi_n\oplus\psi_n)\pi)^\infty_{n=1}$
verifying $(\phi_n\oplus\psi_n)\pi\le c\pi$ for all $n\ge1.$ But
this is not the end of the story.

In the accompanying paper \cite{BLS09b}, rather elaborate
extensions of the above example are analyzed. By means of examples
(which are too long to be recalled here), it is shown that instead
of the functions or, equivalently,  countably additive measures $\hh\pi,$ one has to consider
finitely additive measures.  
This might be seen as a consequence
of the limiting behavior of functions $\fop$ tending to $-\infty$
somewhere, under the seemingly contradictory requirement
\eqref{eq-03}.

\section{Existence of a dual optimizer}\label{ExistenceOptiSection}

The remainder of this article is devoted to developing a theory
which makes this circle of ideas precise in the general setting of
Borel measurable cost functions $c:\XY\to [0,\infty]$. To do so we
shall apply Fenchel's perturbation method as in \cite{BLS09a}. In
addition, we need some functional analytic machinery, in
particular we shall use the space $(L^1)^{**}=(L^\infty)^*$ of
finitely additive measures.

Assume $\Pi(\mu,\nu,c)\neq \emptyset$ to avoid the trivial case.

We fix $\pi_0\in\Pi(\mu,\nu,c)$ and stress that we do \emph{not}
assume that $\pi_0$ has minimal transport cost. In fact, there is
little reason in the present setting (where $c$ is not assumed to
be lower semicontinuous) why a primal optimizer $\widehat{\pi}$
should exist. We denote by $\Pi^{(\pi_0)}(\mu,\nu)$ the set of
elements $\pi\in\Pi(\mu,\nu)$ such that $\pi \ll \pi_0$ and
$\big\| \frac{d\pi}{d\pi_0} \big\|_{L^\infty(\pi_0)}<\infty$. Note
that $\Pi^{(\pi_0)}(\mu,\nu)=\Pi(\mu,\nu)\cap L^\infty(\pi_0)
\subseteq \Pi(\mu,\nu,c)$.

We shall replace the usual Kantorovich optimization problem over
the set $\Pi(\mu,\nu,c)$ by the optimization over the smaller set
$\Pi^{(\pi_0)}(\mu,\nu)$ and consider
\begin{align}
 \label{Walters7} P^{(\pi_0)}
 &= \inf \{ \langle c,\pi\rangle =\textstyle{\int} c\, d\pi  : \pi\in\Pi^{(\pi_0)}(\mu,\nu)\}.
\end{align}
As regards the dual problem, we define for $\eps>0$,
\begin{equation*}\begin{split}
    D^{(\pi_0,\eps)}
  = \sup\Big\{\int \phi\,d\mu+\int\psi\,d\nu:\ & \phi\in L^1(\mu),\psi\in L^1(\nu),\\
        & \int_{\XY} (\fop -c)_+\,d\pi_0\le\eps
        \Big\}\quad \textrm{and}
\end{split}
\end{equation*}
\begin{equation}\label{Walters7D}
    D^{(\pi_0)}=\lim_{\eps\to 0}D^{(\pi_0,\eps)}.
\end{equation}

Define the ``summing'' map $S$ by
\begin{align*}
S:  L^1(X,\mu) \times L^1(Y,\nu) &\to L^1 (\XY,\pi_0)\\
(\phi,\psi) &\mapsto \phi\oplus\psi
\end{align*}
and denote by $L_S^1(\XY,\pi_0)$ the $\|.\|_1$-closed linear
subspace of $L^1(\XY,\pi_0)$ spanned by $S(L^1(X,\mu)\times
L^1(Y,\nu))$. Clearly $L_S^1(\XY,\pi_0)$ is a Banach space under
the norm $\|.\|_1$ induced by $L^1(\XY,\pi_0)$.

We shall also need the bi-dual $L_S^1(\XY,\pi_0)^{**}$ which may
be identified with a subspace of $L^1(\XY,\pi_0)^{**}$. In
particular, an element $h\in L_S^1(\XY,\pi_0)^{**}$ can be
decomposed into $h=h^r+h^s,$ where $h^r\in L^1(\XY,\pi_0) $ is the
regular part of the finitely additive measure $h$ and $h^s $ its
purely singular part. Note that it may happen that $h\in
L_S^1(\XY,\pi_0)^{**}$ while $h^r\not\in L^1_S(\XY,\pi_0),$ and
therefore also $h^s\not\in L_S^1(\XY,\pi_0)^{**}.$

\begin{theorem}\label{LeonardDuality}
Let $c: \XY \to [0,\infty]$ be Borel measurable and let
$\pi_0\in\Pi(\mu,\nu,c)$ be a finite transport plan.
We have
\begin{align}\label{NiceEq}
P^{(\pi_0)}=D^{(\pi_0)}.
\end{align}
There is an element $\hh\in L_S^1(\XY,\pi_0)^{**}$ which verifies
the inequality\footnote{The inequality $\hh\leq c$ pertains to the
lattice order of $L^1(\XY)^{**}$, where we identify the
$\pi_0$-integrable function $c$ with an element of
$L^1(\XY,\pi_0)^{**}.$ If $\hh$ decomposes into $\hh=\hh^r+\hh^s$,
the inequality $\hh\leq c$ holds true if and only if $\hh^r(x,y)
\leq c(x,y)$, $\pi_0$-a.s.\ and $\hh^s\leq 0$ (compare the
discussion after \eqref{Walters8})} $\hh\leq c$ in the Banach
lattice $L^1(\XY,\pi_0)^{**}$ and
$$D^{(\pi_0)}=\langle \hh, \pi_0\rangle.$$
If $  \pi \in \Pi^{(\pi_0)}(\mu,\nu)$ (identifying $ \pi$ with
$\frac{d  \pi}{d\pi_0}$) satisfies $\int c\, d \pi \le P^{(\pi_0)}
+\alpha$ for some number $\alpha \geq 0$, then
\begin{equation}\label{eq-04}
    -\alpha\le\langle \hh^s,\pi\rangle \leq 0.
\end{equation}
In particular, if $  \pi$ is an optimizer of (\ref{Walters7}),
then $\hh^s $ vanishes on the set $\{\frac{d \pi}{d\pi_0}> 0\}$.
\\
In addition, we may find a sequence of elements
$(\phi_n,\psi_n)\in L^1(\mu)\times L^1(\nu)$  such that
\begin{align}
&\fopn\to \hh^r,\ \pi_0\mbox{-a.s.},\nonumber\\
&\|(\fopn-\hh^r)_+\|_{L_1(\pi_0)}\to 0\quad \textrm{and}\nonumber\\
& \lim_{\delta\to0}\sup_{A\subseteq \XY,\pi_0(A)<\delta}\lim_{n\to
\infty} -\langle (\fopn)\1_A,\pi_0\rangle =
\|\hh^s\|_{L_1(\pi_0)^{**}}. \label{Walters9}
\end{align}
\end{theorem}

\begin{proof}
It is straightforward to verify the trivial duality relation
$D^{(\pi_0)}\leq P^{(\pi_0)}.$ To show the reverse inequality  and
to find the dual optimizer $\hh\in L^1(\XY,\pi_0)^{**}$, as in
\cite{BLS09a} we apply W.~Fenchel's perturbation argument. (For an
elementary treatment, compare also \cite{BLS09b}.) The summing map
$S$ factors through $L_S^1(\pi_0)$ as indicated in the subsequent
diagram:
\begin{eqnarray*}
L^1(\mu) \times L^1(\nu) &\stackrel{S}{\longrightarrow}& L^1 (\pi_0)
\\ \stackrel{~~~S_1}{\searrow} & & \stackrel{S_2~~~}{\nearrow}
\\ & L_S^1(\pi_0) &
\end{eqnarray*}
Then $S_1$ has dense range and $S_2$ is an isometric
embedding.  Denote by $\big( L_S^1(\pi_0)^\ast,
\|.\|_{L_S^1(\pi_0)^\ast} \big)$ the dual of $L_S^1(\pi_0)$ which
is a quotient space of $L^\infty(\pi_0)$. Transposing the above
diagram we get
\begin{eqnarray*}
L^\infty(\mu) \times L^\infty(\nu) &\stackrel{T}{\longleftarrow}& L^\infty (\pi_0)
\\ \stackrel{~~~T_1}{\nwarrow} & & \stackrel{T_2~~~}{\swarrow}
\\ & L_S^1(\pi_0)^\ast &
\end{eqnarray*}
where $T,T_1, T_2$ are the transposed maps of $S,S_1,$ resp.\
$S_2$. Clearly $T(\gamma) = (p_X(\gamma), p_Y(\gamma))$ for
$\gamma\in L^\infty(\pi_0)$, where $p_X, p_Y$ are the projections
of a measure $\gamma$ (identified with the
Radon-Nikodym-derivative $\frac{d\gamma}{d\pi_0}$) onto its
marginals. By elementary duality relations we have that $T_2$ is a
quotient map and $T_1$ is injective; the latter fact
allows us to identify the space $L_S^1(\pi_0)^\ast$ with a subspace of $L^\infty(\mu)\times L^\infty(\nu)$. \\
For example, consider the element $\I\in L^\infty(\pi_0)$, which
corresponds to the measure $\pi_0$ on $\XY$. The element $T_2(\I)
\in L_S^1(\pi_0)^\ast$ may then be identified with the element
$(\I,\I)=T(\I)$ in $L^\infty(\mu) \times L^\infty(\nu)$ which
corresponds to the pair $(\mu,\nu)$. We take the liberty to
henceforth denote this element simply by $\I$, independently of
whether we consider  it as an element of $L^\infty(\pi_0)$,
$L_S^1(\pi_0)^\ast$ or $L^\infty(\mu)\times L^\infty(\nu)$.

We may now rephrase the primal problem (\ref{Walters7}) as
\[
\langle c,\gamma\rangle  = \int_{\XY} c(x,y) \, d \gamma(x,y) \to
\min, \quad \gamma\in L_+^\infty(\pi_0),
\]
under the constraint
\begin{equation}
T(\gamma)=\I. \label{Sandra9}
\end{equation}
The decisive trick is to replace (\ref{Sandra9}) by the trivially
equivalent constraint
\[
T_2(\gamma)=\I,
\]
and to perform the Fenchel perturbation argument \emph{not} in the
space $L^\infty(\mu) \times L^\infty(\nu)$ but rather in the
subspace $L_S^1(\pi_0)^\ast$ which is endowed with a
\emph{stronger norm}. The map $\Phi$: $L_S^1(\pi_0)^\ast \to
[0,\infty]$,
\[
\Phi(p) := \inf \{ \langle c,\gamma\rangle  : \gamma \in
L_+^\infty(\pi_0), T_2(\gamma)=p \}, \quad p\in L^1_S(\pi_0)^*,
\]
is  convex, positively homogeneous and $\Phi(\I)=P^{(\pi_0)}$.

\smallskip
\noindent\textbf{Claim.}\ \textit{There is a neighbourhood $V$ of
$\I$ in $L_S^1(\pi_0)^\ast$ on which $\Phi$ is bounded. }
\\
Indeed, let $U=\big\{ \gamma \in L^\infty(\pi_0) \mid
\|\gamma-\I\|_{L^\infty(\pi_0)} <\frac{1}{2} \big\}$. Then $U$ is
contained in the positive orthant $L_+^\infty(\pi_0)$ of
$L^\infty(\pi_0)$ and
\begin{align*}
\Phi(T_2(\gamma))\leq \langle c,\gamma\rangle  \leq \tfrac{3}{2}
\|c\|_{L^1(\pi_0)} \mbox{ for all }\gamma\in U.
\end{align*}
 Hence on $T_2(U)$, which simply is the open ball of radius $\frac{1}{2}$ around $\I$ in the Banach space
$L_S^1(\pi_0)^\ast$, we have that $\Phi$ is bounded by $\frac{3}{2}\|c\|_{L^1(\pi_0)}$.

\medskip

It follows from elementary geometric facts that the convex
function $\Phi$ is continuous on $T_2 (U)$ with respect to the
norm of $L^1_S(\pi_0)^*.$ By Hahn-Banach there exists $f\in
L^1_S(\pi_0)^{\ast\ast}$ such that
\begin{align*}
\langle f, \I\rangle & = \Phi(\I),\\
\langle f, p\rangle & \leq \Phi (p) \mbox { for all } p\in L^1_S(\pi_0)^{\ast}.
\end{align*}

The adjoint $T^*_2$ of $T_2$ maps $L_S^1(\pi_0)^{\ast\ast}$
isometrically onto a subspace $E$ of $L^1(\pi_0)^{\ast\ast}=
L^\infty(\pi_0)^\ast$. The space $E$ consists of those elements of
$L^1(\pi_0)^{\ast\ast}$ which are $\sigma^\ast$-limits of nets
$(\phi_\alpha\oplus\psi_\alpha)_{\alpha\in I}$ with $\phi_\alpha
\in L^1(\mu)$, $\psi_\alpha\in L^1(\nu)$. Write $\hh:=T_2^*(f)$.
Then for all $ \gamma\in L_+^\infty(\pi_0)$,
\begin{align}\label{TrivialPart}
\langle \hh, \gamma\rangle= \langle T_2^*(f),\gamma\rangle
=\langle f, T_2(\gamma)\rangle\ \leq\ \Phi(T_2(\gamma))\ \leq\
\langle c,\gamma\rangle,
\end{align}
and if $\pi\in L_+^\infty(\pi_0), T_2(\pi)=\I$ then
\begin{align}\label{AttainmentPart}
\langle \hh,\pi\rangle= \langle T_2^*(f),\pi\rangle =\langle f,
T_2(\pi)\rangle=\langle f, \I \rangle =\Phi(\I)=P^{(\pi_0)}.
\end{align}
By (\ref{TrivialPart}), the inequality $\hh\leq c$ holds true in
the Banach-lattice $L^\infty(\pi_0)^*$. Combining this with
(\ref{AttainmentPart}) we obtain that $\hh$ is a dual optimizer in
the sense of
\begin{equation}\label{Walters8}
\begin{split}
   D_{**}^{(\pi_0)}:= \sup \big\{ \langle g,\pi_0\rangle:\ & g\in L_S^1
(\pi_0)^{**},\,  g\leq c\\
        & \mbox{ in the Banach lattice } L^1(\pi_0)^{**} \big\}
\end{split}
\end{equation}
(where we identify $\pi_0$ with the element $\I$ of $L^\infty(\pi_0)$)
and that there is no duality gap in this sense, i.e.\
$
    D_{**}^{(\pi_0)}=P^{(\pi_0)}.
$

As mentioned above, every element $g \in L^\infty(\pi_0)^*$ splits
in a regular part $g^r$ lying in $L^1(\pi_0)$ and a purely
singular part $g^s.$ Given $g_1,g_2\in L^\infty(\pi_0)^*$, we have
$g_1\leq g_2$ if and only if $g_1^r\leq g_2^r$ and $g_1^s\leq
g_2^s$. Since $c\in L^1(\pi_0)$ we have $c^s=0$. The inequality
$\hh\leq c$ implies that $\hh^s\leq c^s=0$ and $\hh^r\leq c^r=c$.
It follows that for each $\pi\in L_+^\infty(\pi_0)$
\begin{align}\label{LessThan}
\langle \hh^r, \pi\rangle \leq \langle c,\pi\rangle.
\end{align}
Assume additionally that $\pi$ satisfies $T_2(\pi)=\I$ and choose
$\alpha\geq 0$ such that $\langle c, \pi \rangle \leq
P^{(\pi_0)}+\alpha$. Then $\langle \hh, \pi\rangle=P^{(\pi_0)}$
and subtracting this quantity from (\ref{LessThan}) we get
\begin{align*}
\langle -\hh^s, \pi\rangle = \langle \hh^r-\hh, \pi\rangle\leq
\langle c,\pi\rangle - P^{(\pi_0)}\leq \alpha
\end{align*}
showing \eqref{eq-04}.

We still have to show the existence of a \emph{sequence}
$(\phi_n,\psi_n)_{n=1}^\infty$ satisfying the above assertions
about convergence. So far we know that there is a net
$(\phi_\alpha, \psi_\alpha)_{\alpha\in I}$ such that
$\phi_\alpha\oplus \psi_\alpha$ weak-star converges to $\hh$.
First we claim that there exists a net $(f_\alpha)_{\alpha\in I}$
of elements of $L^1(\pi_0)$, such that $\|f_\alpha\|_1\leq \|
\hh^s\|$, $\hh^r+f_\alpha \in L_S^1(\pi_0)$ and $\hh^r+f_\alpha\to
\hh$ in the $\sigma^*$-topology. To see this, note that Alaoglu's
theorem \cite[Theorem IV.21]{RS1} implies that in a Banach space
$V$, the unit ball $B_1(V)$ is $\sigma^*$-dense in the unit ball
$B_1(V^{**})$ of the bidual. Thus $\hh^r+ \|\hh^s\|
B_1(L_S^1(\pi_0))$ is $\sigma^*$-dense in $\hh^r+ \|\hh^s\|
B_1(L_S^1(\pi_0)^{**})$ which yields the existence of a net
$(f_\alpha)_{\alpha\in I}$ as
required. 

\medskip
As $\hh^s$ is purely singular, we may find a sequence
$(\alpha_n)^\infty_{n=1}$ in $I$ such that 
$\|f_{\alpha_n}\|\leq \|\hh^s\|$ and  $ \int f_{\alpha_n}\,
d\pi_0=-\|\hh^s\|+2^{-n}$, and that $\int(|f_{\alpha_n}|\wedge
2^n)\,d\pi_0\le 2^{-n}$, which implies that the sequence
$(f_{\alpha_n})^\infty_{n=1}$ converges $\pi_0$-a.s.\ to zero.

As $\hh^r +f_{\alpha_n} \in L^1_S(\pi_0)$ we may find
$(\phi_n,\psi_n)\in L^1(\mu)\times L^1(\nu)$ such that
$$
\|\phi_n\oplus\psi_n -(\hh^r +f_{\alpha_n})\|_{L^1(\pi_0)}
<2^{-n}.
$$
We then have that $(\fopn)^\infty_{n=1}$ converges $\pi_0$-a.s.\
to $\hh^r$ and that $\|(\fopn-\hh^r)_+\|_{L^1(\pi_0)} \to 0$.

As regards assertion (\ref{Walters9}) 
we note that, for $A_m=\bigcup^\infty_{n=m+1}\{|f_{\alpha_n}|
>2^{-n}\}$ we have $\pi_0 (A_m)\le 2^{-m}$ and
\begin{align*}
\liminf\limits_{n\to \infty} (-\langle
(\phi_n\oplus\psi_n)\1_{A_m},\pi_0\rangle)
&  = -\limsup\limits_{n\to\infty} \langle(\hh^r+f_{\alpha_n}) \1_{A_m},\pi_0\rangle \\
&  = -\langle \hh^r\1_{A_m},\pi_0\rangle -\lim\limits_{n\to\infty} \langle f_{\alpha_n}\1_{A_m},\pi_0\rangle  \\
&  = -\langle \hh^r\1_{A_m},\pi_0\rangle +\|\hh^s\|_{L^1(\pi_0)}
{**}.
 \end{align*}
Letting $m$ tend to infinity we obtain that the left hand side of
(\ref{Walters9}) is greater than or equal to the right hand side.
As regards the reverse inequality it suffices to note that
$\|f_{\alpha_n}\|_{L^1(\pi_0)} \le \|\hh^s\|_{L^1(\pi_0)^{**}}$.

As $\hh^r\leq c, \pi_0$-a.s., we obtain in particular that
$\|(\phi_n\oplus \psi_n-c)_+\|_{L^1(\pi_0)}\to 0$ showing that
$D^{(\pi_0)}\geq P^{(\pi_0)}$ and therefore \eqref{NiceEq}, the
reverse inequality being straightforward.
\end{proof}

As a by-product of this proof, we have shown in \eqref{Walters8}
that
\begin{equation}\label{eq-05}
    D_{**}^{(\pi_0)}=D^{(\pi_0)}=P^{(\pi_0)}.
\end{equation}

Admittedly, Theorem \ref{LeonardDuality} is rather abstract.
However, we believe that it may be useful in applications to have
the possibility to pass to {\it some kind of limit $\hh$} of an
optimizing sequence $(\phi_n,\psi_n)^\infty_{n=1}$ in the dual
optimization problem, even if this limit is somewhat awkward. To
develop some intuition for the message of Theorem
\ref{LeonardDuality}, we shall illustrate the situation at the
hand of some examples.

Let us start with Example \ref{SecondAP}. In this case we may
apply Theorem \ref{LeonardDuality} to the finite transport plan
$\pi_{\frac{1}{2}} =\frac{1}{2} (\pi_0 +\pi_1)$, (we apologize for
using $\pi_{\frac{1}{2}}$ instead of $\pi_0$ in Theorem
\ref{LeonardDuality} as the notation $\pi_0$ is already taken). As
we have seen above, there are sequences
$(\phi_n\oplus\psi_n)^\infty_{n=1}$ converging
$\pi_{\frac{1}{2}}$-a.s.\ as well as in the norm of
$L^1(\pi_{\frac{1}{2}})$ to $\hh=c$, as defined in Example
\ref{SecondAP} above. In particular we do not have to bother about
the singular part $\hh^s$ of $\hh$, as we have $\hh=\hh^r$ in this
example. We find again that $h$ represents the limit of
$(\phi_n\oplus \psi_n)^\infty_{n=1}$, considered as a Borel
function on $\{c <\infty\}$ which is the support of
$\pi_{\frac12}.$

\medskip

We now make the example a bit more interesting and challenging. (See Example \ref{Example3.3} below.)

Fix in the context of Example \ref{SecondAP} (where we now write
$\tilde{c}$ instead of $c$ to keep the letter $c$ free for a new
function to be constructed) a sequence
$(\phi_n,\psi_n)^\infty_{n=1}$ such that $\|\tilde{c}-
\phi_n\oplus\psi_n\|_{L^1(\pi_i)}\to 0$ for $i=0,1$. We claim that
$(\phi_n\oplus\psi_n)^\infty_{n=1}$ converges in
$\|.\|_{L^1({\pi_k})}$ where, for each $k\in\N$, $\pi_k$ is the
measure which is uniformly distributed on
\begin{align}
\label{G37}
\Gamma_k=\{(x,x\oplus k\alpha): x\in [0,1)\}.
\end{align}
Let us prove  this convergence whose precise statement is given
below at \eqref{DefinitionOfD} and \eqref{ConcreteBound}. We know
that\footnote{The equations (\ref{ToCEq}) to
(\ref{TheLastEquation}) refer to integrable functions on $[0,1)$
and convergence is understood to be with respect to
$\|.\|_{L^1(\mu)}$.}
\begin{align}
    &\phi_n(x)+\psi_n(x) \quad \to  \quad \tilde c(x,x) \mbox { and} \label{ToCEq} \\
    &\phi_n(x)+\psi_n(x\oplus \alpha) \quad \to \quad \tilde c(x,x\oplus \alpha), \mbox{whence }\nonumber \\
    &\psi_n(x\oplus \alpha)-\psi_n(x) \quad \to \quad\underbrace{\tilde c(x,x\oplus \alpha)-\tilde c(x,x)}_{=:g(x)}=\left\{
        \begin{array}{ll}+ 1& \mbox{ for $x\in [0,\tfrac12)$},\\
                        -1& \mbox{ for $x\in [\tfrac12,1)$}.
        \end{array}\right. \label{DifferenceEquation}
\end{align}
Replacing $x$ by $x\oplus i\alpha, \ i=1,\ldots, k-1$ in (\ref{DifferenceEquation}) this yields
\begin{align*}
\psi_n(x\oplus \alpha)-\psi_n(x)\to\sum_{i=0}^{k-1} g(x\oplus i
\alpha).
\end{align*}
Combined with (\ref{ToCEq}) we have
\begin{align}\label{TheRelevantConvergence}
&\quad\ \lim_{n\to\infty} [\phi_n(x)+\psi_n(x\oplus k\alpha)] =  1+ \sum_{i=0}^{k-1} g(x\oplus i \alpha)\\
   &=\ {1+\ \#\left\{0\leq i<k: x\oplus i\alpha\in  [0,\tfrac12) \right\}-\  \#\left\{0\leq i<k: x\oplus i\alpha\in [\tfrac12,1)\right\}}\nonumber \\ \label{TheLastEquation}
   &=:\ \rho_k(x).
\end{align}
Define the function $h$ on $\XY$
\begin{align}\label{DefinitionOfD}
h(x,y)=\left\{
\begin{array}{cl}
\rho_k(x)&\mbox{ for }(x,y)\in\Gamma_k, k\in \N,\\
\infty&\mbox{ else}.
\end{array}\right.
\end{align}
By (\ref{TheRelevantConvergence}), we have, for each $k\in \N$,
    $\lim_n\|h-\phi_n\oplus\psi_n\|_{L^1(\pi_k)}=0$. Somewhat more precisely, one obtains
that
\begin{align}\label{ConcreteBound}
\|h-\phi_n\oplus\psi_n\|_{L^1(\pi_k)}\leq  k \|\tilde
c-\phi_n\oplus\psi_n\|_{L^1(\pi_0+\pi_1)}.
\end{align}

Now we shall  modify the cost function $\tilde c$ of Example
\ref{SecondAP} by defining it to be finite not only on $\Gamma_0
\cup\Gamma_1$, but rather on $\bigcup_{k\in\N} \Gamma_k$. We then
obtain the following situation.

\begin{example}\label{Example3.3}
Using \eqref{DefinitionOfD} define
$c:[0,1)\times[0,1)\to[0,\infty]$ by
$$c(x,y)=h(x,y)_+,$$
so that $\{c<\infty\}=\bigcup_{k\in\N} \Gamma_k$. For the resulting optimal transport problem we then find:
\begin{enumerate}
    \item[(i)] The primal value $P$ of the problem (\ref{G1}) equals zero
    and $\hat{\phi}=\hat{\psi}=0$ are (trivial) optimizers of the dual problem (\ref{SimpleJ}).
    \item[(ii)] For strictly positive scalars $(a_k)_{k\geq0}$, normalized by
$\sum_{k\geq0} a_k=1$ apply Theorem \ref{LeonardDuality} to the
transport plan $\pi:=\sum_{k\geq0} a_k\pi_k.$ (Again we apologize
for using the notation $\pi$ for the measure $\pi_0$ in Theorem
\ref{LeonardDuality}, as all the letters $\pi_k$ are already
taken.) If $(a_k)_{\geq0}$ tends sufficiently fast to zero, as
$|k|\to\infty,$ the following facts are verified.
\begin{enumerate}
    \item[-] The primal value is
    $$P^{(\pi)} =\inf\left\{\int_{\XY} c\,d\bar \pi :\bar{\pi} \in\Pi (\mu,\nu),\|\tfrac{d\bar\pi}{d\pi}\|_{L^\infty} <\infty\right\}=1.$$
    \item[-] The Borel function $h\in L^1(\pi)$ defined in
\eqref{DefinitionOfD} is a dual optimizer in the sense of Theorem
\ref{LeonardDuality}, i.e.
$$D^{(\pi)}=\int_{\XY} h\,d\pi=1.$$
    \item[-] There is a sequence $(\phi_n,\psi_n)^\infty_{n=1}$ in
$L^1(\mu)\times L^1(\nu)$ such that $(\phi_n\oplus
\psi_n)^\infty_{n=1}$ converges to $h$ in the norm of $L^1(\pi)$.
\end{enumerate}
\end{enumerate}
\end{example}

Before proving the above assertions let us draw one conclusion: in
(ii) we {\it can not assert} that the functions
$(\phi_n,\psi_n)^\infty_{n=1}$ satisfy -- in addition to the
properties above -- the inequality $\phi_n(x)+\psi_n(y) \le
c(x,y)$, for all $(x,y)\in \XY$. Indeed, if  this were possible
then, because of $\lim_{n\to\infty} (\int_X \phi_n \,d\mu+\int_Y
\psi_n \,d\nu)=D^{(\pi)}=1,$ we would have that the dual value $D$
of the original dual problem \eqref{SimpleJ} would equal $D=1$, in
contradiction to (i).

\begin{proof}[Proof of the assertions of Example \ref{Example3.3}]

We start with assertion (ii). Fix an optimizing sequence
$(\phi_n,\psi_n)_{n=1}^\infty$ in the context of Example
\ref{SecondAP} such that
\begin{align}
\label{ForConcrete}\|\tilde{c}-
\phi_n\oplus\psi_n\|_{L^1(\pi_0+\pi_1)}\leq {1}/{n^3}.
\end{align}
 Pick a sequence  $(a_k)_{k\in\N}$ of positive numbers such that
\begin{itemize}
\item[(a)] $a_k \|h\|_{L^1(\pi_k)}\leq C 2^{-k} $ for all $k\in
\N$, \item[(b)] $a_k (\|\phi_n\|_1+\|\psi_n\|_1)\leq C  2^{-k}$
for  all $k\in \N$ with $n\leq k,$
\end{itemize}
for    some real constant $C$.
After re-normalizing, if necessary, we may assume that $\sum_{k= 1
}^\infty a_k=1$. Set $\pi:= \sum_{k=1}^\infty a_k \pi_k$. From (a) we
obtain $h\in L^1(\pi)\subseteq L^1(\pi)^{**}$ thus $h$ is viable
for the problem $D^{(\pi)}_{**}$ and hence $D^{(\pi)}_{**}\geq 1$.
Clearly $P^{(\pi)}\leq1$, hence $P^{(\pi)}=D^{(\pi)}_{**}=1$ and
$h$ is a dual maximizer. Combining \eqref{ForConcrete} with
\eqref{ConcreteBound} we obtain
\begin{align*}
\|h-\phi_n\oplus\psi_n\|_{L^1(\pi_k)}\leq  k /n^3.
\end{align*}
Therefore
\begin{align*}
\|h-\phi_n\oplus\psi_n\|_{L^1(\pi)} &\leq   \sum_{k\leq
n}\|h-\phi_n\oplus\psi_n\|_{L^1(\pi_k)}
+ \sum_{k>n} a_k(\|h\|_{L^1(\pi_k)} + \|\phi_n\|_1+\|\psi_n\|_1)\\
&\leq  1/n +2C \sum_{k>n} 2^{-k}.
\end{align*}
Hence $\phi_n\oplus\psi_n$ converges to $h$ in $\|.\|_{L^1(\pi)}.$
 This shows
assertion (ii) above.

\medskip

To obtain (i) we  construct a transport plan $\pi_\beta
\in\Pi(\mu,\nu)$ such that $\int_{\XY} c\, d\pi_\beta=0$. Note in
passing that in view of (ii) we must have
$\|\tfrac{d\pi_\beta}{d\pi}\|_{L^\infty (\pi)}=\infty$ for the
$\pi$ constructed above. On the other hand, we must have
$\tfrac{d\pi_\beta}{d\pi} \in L^1 (\pi)$, if $a_k>0$ for all
$k\in\N$, as every finite cost transport plan must be absolutely
continuous with respect to $\pi$.

The idea is to concentrate $\pi_\beta $ on the set
\begin{align*}
\Gamma&:=\{(x,y):c(x,y)=0\}\\
&=\{(x,x\oplus k\alpha): k\geq 1, \textstyle{ \sum_{i=0}^{k-1}
(\1_{[0, \frac12)}(x\oplus i\alpha)-\1_{[\frac12, 1)}(x\oplus
i\alpha))\leq -1}\}.
\end{align*}
To prove that this can be done it is sufficient to show that whenever $A\subseteq X$, $B\subseteq Y,$ $\mu (A), \nu (B)>0$, a
subset $A'$ of $A$ can be transported to a subset $B'$ of $B$ with $\nu (B')=\mu(A')>0$ via $\Gamma$.  Then an exhaustion argument applies.

At this stage we encounter an interesting connection to the theory of measure preserving systems.
For $x\in X$ and $m\in \N$ set
\begin{align*}
S(x,m):= \Big(x\oplus \alpha, m+ \1_{[0,
\frac12)}(x)-\1_{[\frac12, 1)}(x)\Big).
\end{align*}
Then $S$ is a measure preserving transformation of the space
$([0,1]\times\Z, \lambda\times \#)$. (See \cite{Aaro97} for an
introduction to infinite ergodic theory and the basic definitions
in this field.) It is not hard to see that the ergodic theorem,
applied to the rotation by $\alpha$ on the torus, shows that $S$ is
non wandering. Much less trivial is the fact that $S$ is also
ergodic. This was shown by K.\ Schmidt \cite{Schm78} for a certain
class of irrational numbers $\alpha\in [0,1)$, and in full
generality by M.\ Keane and J.-P.\ Conze \cite{CoKe76}, see also
\cite{AaKe82}.
\\
The relevance of these facts to our situation is that for $k\geq 1$, the pair $(x, x\oplus k\alpha)$ is an element of $\Gamma$ if and only if
$S^k(x,0)\in [0,1)\times \{-1,-2,\ldots\}$. By ergodicity of $S$, there exists $k$ such that
\begin{align*}
(\lambda \times \#)\big((S^k[A\times \{0\}]) \cap (B\times \{-1,-2,\ldots\})\big)>0,
\end{align*}
 thus it is
possible to shift a positive portion of $A$ to $B$ as required. By exhaustion, there indeed exists a transport $\pi_\beta $ such that $\langle c,\pi_\beta\rangle=0$.
\end{proof}

The above example 
 illustrates some of the
subtleties of Theorem \ref{LeonardDuality}. However, it does not
yet provide evidence for the necessity of allowing for the
singular part $\hh^s$ of the optimizer $\hh$ in Theorem
\ref{LeonardDuality}. We have constructed yet a more refined --
and rather longish -- variant of the Ambrosio--Pratelli example
above, which shows that, in general, there is no way of avoiding
these complications in the statement of Theorem
\ref{LeonardDuality}. We refer to the accompanying paper
\cite[Section 3]{BLS09b} for a presentation of this example, where
it is shown that it can indeed occur that the singular part
$\hh^s$ in Theorem \ref{LeonardDuality} does not vanish.

\section{The Projective Limit Theorem}\label{sec-projlim}

We again consider the general setting where $c$ is a
$[0,\infty]$-valued Borel measurable function. To avoid
trivialities we shall always assume that $\Pi(\mu,\nu,c)$ is
non-empty.

Theorem \ref{LeonardDuality} only pertains to the situation of a
{\it fixed} element $\pi_0\in\Pi(\mu,\nu,c)$: one then optimizes
the transport problem of all ${\pi}\in\Pi(\mu,\nu)$ with
$\|\frac{d{\pi}}{d\pi_0}\|_{L^\infty (\pi_0)} <\infty$.

The purpose of this section is to find an optimizer $h$ which does
work simultaneously, {\it for all $\pi_0\in\Pi(\mu,\nu,c)$}. We
are not able to provide a result showing that a {\it function} $h$
-- plus possibly some singular part $h^s$ -- exists which fulfills
this duty, for all $\pi_0\in\Pi(\mu,\nu,c)$. We have to leave the
question whether this is always possible as an open problem. But
we can show that a projective limit $\hat{H} =(\hh_\pi
)_{\pi\in\Pi(\mu,\nu,c)}$ exists which does the job.

We introduce an order relation on $\Pi(\mu,\nu,c):$ we say that
$\pi_1 \preceq\pi_2$ if $\pi_1 \ll\pi_2$ and $\|
\frac{d\pi_1}{d\pi_2}\|_{L^\infty (\pi_2)} <\infty.$ For
$\pi_1\preceq\pi_2$ there is a natural, continuous projection
$P_{\pi_1 ,\pi_2}: L^1(\pi_2)\to L^1(\pi_1)$ associating to each
$h_{\pi_2}\in L^1(\pi_2)$, which is an equivalence class modulo
$\pi_2$-null functions, the equivalence class modulo $\pi_1$-null
functions which contains the equivalence class $h_{\pi_2}$ (and
where this inclusion of equivalence classes may be strict, in
general). We may define the locally convex vector space $E$ as the
projective limit
$$
E= \underset\longleftarrow\lim_{\pi\in\Pi (\mu,\nu,c)} \quad
L^1(\XY, \pi).
$$
The elements of $E$ are families
$H=(h_\pi)_{\pi\in\Pi(\mu,\nu,c)}$ such that, for $\pi_1\preceq
\pi_2$, we have $P_{\pi_1,\pi_2} (h_{\pi_2}) =h_{\pi_1}.$
\\
A net $(H^\alpha)_{\alpha\in I}\in E$ converges to $H\in E$ if,
$$
\lim_{\alpha\in I} \parallel h^\alpha_\pi -h_\pi \parallel_{L^1(\pi)} =0, \quad \mbox{for each} \ \pi\in\Pi(\mu,\nu,c).
$$
We may also define the projective limit
$$E_S =  \underset\longleftarrow\lim_{\pi\in\Pi (\mu,\nu,c)} L^1_S (\XY,\pi),$$
which is a closed subspace of $E$.

We start with an easy result.
\begin{proposition}\label{AggregatingCountablyManyPis}
Let $X$ and $Y$ be polish spaces equipped with Borel probability
measures $\mu,\nu$, and let $c:\XY\to [0,\infty]$ be Borel
measurable. Assume that $\Pi(\mu,\nu,c)$ is non-empty.

There is $\pi_0\in\Pi(\mu,\nu,c)$ such that
$$P^{(\pi_0)} = \inf\limits_{\pi\in\Pi (\mu,\nu,c)} P^{(\pi)}.$$
\end{proposition}

\begin{proof}
Let $(\pi_n)^\infty_{n=1}$ be a sequence in $\Pi(\mu,\nu,c)$ such that
$$\lim_{n\to\infty} P^{(\pi_n)} = \inf_{\pi\in\Pi (\mu,\nu,c)} P^{(\pi)}.$$
It suffices to define $\pi_0$ as
$$\pi_0 =\sum^\infty_{n=1} 2^{-n} \ \pi_n$$
as we then have $\pi_n\preceq\pi_0$, for each $n\in\mathbb{N}.$
\end{proof}

Of course, if the primal problem \eqref{G1} is attained, we have
$P^{(\pi_0)} =P.$

The above proposition allows us to suppose w.l.o.g.\ in our
considerations on the projective limit $E$ that the $\pi$
appearing in the definition are all bigger than $\pi_0$:
$$E=  \underset\longleftarrow\lim_{\pi\in\Pi (\mu,\nu,c)} L^1(\pi) =
 \underset\longleftarrow\lim_{\pi\in\Pi (\mu,\nu,c),\pi\succeq\pi_0}
L^1(\pi).$$ Clearly, we then have that the optimal transport cost
$P^{(\pi)}$ is equal to $P^{(\pi_0)}$, for all $\pi\succeq\pi_0$.

\begin{theorem}\label{GeneralDualityLimit}
Let $X$ and $Y$ be polish spaces equipped with Borel probability
measures $\mu,\nu$, and let $c:\XY\to [0,\infty]$ be Borel
measurable. Assume that $\Pi(\mu,\nu,c)$ is non-empty. Let $\pi_0$
be as in Proposition \ref{AggregatingCountablyManyPis}

There is an element
$\hat{H}=(\hh_\pi)_{\pi\in\Pi(\mu,\nu,c),\pi\succeq\pi_0} \in E$
such that, for each $\pi\in\Pi(\mu,\nu,c),\pi\succeq\pi_0,$ the
element $\hh_\pi\in L^1_S(\pi)^{**}$ satisfies $\hh_\pi\le c$ in
the order of $L^1(\pi)^{**}$ and $\hh_\pi$ is an optimizer of the
dual problem  \eqref{Walters8}
$$
\langle \hh_\pi,\pi\rangle
 =D_{**}^{(\pi)} := \sup \{ \langle h,\pi \rangle : h\in L_S^1 (\pi)^{**}, \, h\leq c\}.
$$
We then have that, for each $\pi\in\Pi(\mu,\nu,c),
\pi\succeq\pi_0,$ the decomposition $\hh_\pi = \hh^r_\pi
+\hh^s_\pi$ of $\hh_\pi$ into its regular and singular parts
verifies
\begin{itemize}
    \item[-] $\hh^r_\pi\in L^1_S(\pi)$ and $\hh^r_\pi\le c$ in
    $L^1(\pi)$;
    \item[-]  $\hh^s_\pi\in L^1_S(\pi)^{**}$ and
        $\hh^s_\pi\le 0$ in the space of purely finitely additive measures
which are absolutely continuous with respect to $\pi.$
\end{itemize}

Moreover, for each $\pi\in\Pi(\mu,\nu,c), \pi\succeq\pi_0,$ there
is no duality gap in the sense that
\begin{equation}\label{eq-06}
    D_{**}^{(\pi)}=D^{(\pi)}=P^{(\pi)}=P^{(\pi_0)}
\end{equation}
where $
    D^{(\pi)}:=\lim\limits_{\eps\to 0}\sup\Big\{\int \phi\,d\mu+\int\psi\,d\nu: \phi\in L^1(\mu),\psi\in L^1(\nu),
     \int (\fop -c)_+\,d\pi\le\eps
        \Big\}
$ and $P^{(\pi)}:= \inf \{ \langle c,\pi'\rangle :
\pi'\in\Pi^{(\pi)}(\mu,\nu)\}.$ If in addition the primal problem
\eqref{G1} is attained, for instance if $c$ is lower
semicontinuous, then
    $
    D_{**}^{(\pi)}=D^{(\pi)}=P^{(\pi)}=P.
    $
\end{theorem}

 \begin{proof}
Fix $\pi\in\Pi(\mu,\nu,c), \pi \succeq \pi_0.$ We have seen in Theorem \ref{LeonardDuality} that the set
$$
    K_\pi =\{h \in L_S^1(\pi)^{**} : h\le c, \langle h,\pi \rangle = \langle c,\pi\rangle\}
$$
is non-empty. In addition $K_\pi$ is closed and bounded in $L^1
(\pi)^{**}$ and hence compact with respect to the $\sigma (L^1_S
(\pi)^{**} ,L^1_S (\pi)^* )$-topology.

For $\pi,\pi ' \in \Pi(\mu,\nu,c)$ with $\pi\preceq\pi'$ the set
$$
K_{\pi,\pi'}= P_{\pi,\pi'}(K_{\pi'})
$$
is contained in $K_\pi$ and
still a non-empty $\sigma^*$-compact convex subset of
$L^1(\pi)^{**}$.    
By compactness the following set is $\sigma^*$-compact and non-empty
too:
$$
    K_{\pi,\infty} = \bigcap\limits_{\pi'\succeq\pi} K_{\pi,\pi'}.
$$
We have $K_{\pi,\infty}=P_{\pi,\pi'}(K_{\pi',\infty})$ for $\pi\preceq\pi'$. Hence by Tychonoff's theorem the projective
limit
$$
    \underset\longleftarrow\lim_{\pi\in\Pi (\mu,\nu,c), \pi\succeq\pi_0} K_{\pi,\infty}
$$
of the compact sets $(K_{\pi,\infty})_{\pi\succeq \pi_0}$ is
non-empty, which is precisely the main assertion of the present
theorem.
\\
Finally, \eqref{eq-06} is a restatement of \eqref{eq-05} and when
the primal problem \eqref{G1} is attained, the last series of
equalities follows from $P^{(\pi_0)} =P$.
\end{proof}

Clearly
    $\Prel\le P\le P^{(\pi_0)},$ hence with Theorem \ref{res-01} and \eqref{eq-06} one sees that
    $$D=\Prel\le P\le P^{(\pi_0)}=P^{(\pi)}=D_{**}^{(\pi)}=D^{(\pi)}$$
for every $\pi\in\Pi(\mu,\nu,c)$ such that $\pi\succeq\pi_0.$


\end{document}